\documentclass[a4paper, 12pt]{article}

\usepackage{amsmath, amssymb, amsthm}
\usepackage{authblk}
\usepackage{biblatex, csquotes}
\usepackage{breakurl}
\usepackage[colorlinks=true, breaklinks=true]{hyperref}
\usepackage[margin=3.2cm]{geometry}
\usepackage{newtxtext,newtxmath}

\DeclareMathOperator{\Char}{char}
\DeclareMathOperator{\dist}{dist}

\DeclareMathOperator{\Int}{int}

\theoremstyle{plain}
\newtheorem{thm}{Theorem}[section]
\newtheorem{prop}[thm]{Proposition}

\newtheorem{lemma}[thm]{Lemma}

\theoremstyle{definition}

\theoremstyle{remark}

\numberwithin{equation}{section}

\newcommand{\R}{\mathbb{R}}
\newcommand{\barU}{\overline{U}}

\newcommand{\curlC}{\mathfrak{C}}
\newcommand{\curlS}{\mathfrak{S}}
\newcommand{\curlX}{\mathfrak{X}}
\newcommand{\barcurlX}{\overline{\curlX}}
\newcommand{\tildeF}{\widetilde{F}}
\newcommand{\tildeH}{\widetilde{H}}

\newcommand{\del}{\partial}
\newcommand{\Om}{\Omega}
\newcommand{\barOm}{\overline{\Om}}

\title{Solution branches of nonlinear eigenvalue problems on restricted domains}

\author{Shane Arora%
    \footnote{Email: saro0188@uni.sydney.edu.au}}
\affil{School of Mathematics and Statistics, University of Sydney, NSW 2006,
    Australia}

\begin{document}

\maketitle

\begin{abstract}
We extend bifurcation results of nonlinear eigenvalue problems
from real Banach spaces to any neighbourhood of a given
point. For points of odd multiplicity on these restricted domains, we establish
that the component of solutions through the bifurcation point either is unbounded,
admits an accumulation point on the boundary, or contains an even number of odd
multiplicity points. In the simple multiplicity case, we
show that branches of solutions in the directions of corresponding
eigenvectors satisfy similar conditions on our domains.
\end{abstract}

\paragraph{MSC Classification (2010):} 47J10; 47J15
\paragraph{Keywords:} Nonlinear eigenvalue problem, bifurcation theory, compact
operators.

\section{Introduction}

For Banach spaces $X$ and $Y$, any subset $B$ of $X$ and any function
$G : B \to Y$, we say that $G$ is \textit{compact} (or \textit{completely continuous})
if it is continuous and maps bounded closed subsets of $X$ contained in $B$ to
relatively compact sets. Let $X$ be an arbitrary Banach space, $\curlX = X \times \R$,
$\lambda_0 \in \R$ and $\Om \subseteq \curlX$ a neighbourhood of $(0, \lambda_0)$. 
We consider the \textit{nonlinear eigenvalue problem} on $\Om$ of the form
\begin{equation}  \label{eqn:nonlinear-eigenvalue-problem}
    0 = x - \lambda Kx - H(x, \lambda) =: F(x, \lambda),
\end{equation}
where $K : \curlX \to X$ is a compact linear operator and $H : \Om \to X$ is
compact. We suppose that $H$ is such that the function $h : \Om \to X$ given by
\[
    h(x, \lambda) =
    \begin{cases}
        \|x\|^{-1} H(x, \lambda), & x \neq 0, \\
        0, & x = 0,
    \end{cases}
\]
is continuous. Note that, when $\Om = \curlX$, this condition on $H$ is
equivalent to the conditions on $H$ given in
\cite[p.~487]{rabinowitzGlobalResultsNonlinear1971}
and \cite[p.~1069]{dancerStructureSolutionsNonlinear1973}.

We say that $(0, \lambda_0) \in \Om$
is a \textit{bifurcation point} of $F(x, \lambda) = 0$ (with respect to the
``curve'' of trivial solutions $x = 0$) if every neighbourhood of
$(0, \lambda_0)$ contains a non-trivial solution of $F(x, \lambda) = 0$. 
It is well-known that if $(0, \lambda_0)$ is a bifurcation point, then
$\lambda_0^{-1}$ is an eigenvalue of $K$
\cite[Proposition~28.1]{deimlingNonlinearFunctionalAnalysis1985}. This motivates
the definition of a \textit{characteristic value} of $K$: any $\lambda \in \R$
such that $\lambda^{-1}$ is an eigenvalue of $K$. We denote the set of
characteristic values of $K$ by $\Char(K)$.

We take the \textit{multiplicity} of a
characteristic value $\lambda_0$ to be the algebraic multiplicity of
$\lambda_0^{-1}$ as an eigenvalue of $K$. It was proved in the pioneering paper
by Rabinowitz
\cite[Theorem~1.3]{rabinowitzGlobalResultsNonlinear1971} that if $\lambda_0$
is of odd multiplicity, then $(0, \lambda_0)$ is a bifurcation point. Moreover,
assuming that $\Om = \curlX$ he showed that, for such $\lambda_0$, the connected
component $C_{\lambda_0}$ containing $(0, \lambda_0)$ of the closure of
non-trivial solutions to $F(x, \lambda) = 0$ either is unbounded or contains
some $(0, \mu) \neq (0, \lambda_0)$, where $\mu \in \Char(K)$ is of odd
multiplicity. A strengthened version of this result
by Dancer \cite[Corollary~1]{dancerStructureSolutionsNonlinear1973} states that
if $C_{\lambda_0}$ is bounded, then it contains an even number of $(0, \mu)$
with $\mu \in \Char(K)$ of odd multiplicity.

Of particular importance is the special case when $\lambda_0$ is of
multiplicity $1$ (i.e. it is \textit{simple}) with corresponding eigenvector $v$.
Then we can express $C_{\lambda_0}$ as the union of
$C_{\lambda_0}^+$ and $C_{\lambda_0}^-$:
closures of the unions of all \textit{branches of solutions} going
from $(0, \lambda_0)$ in the directions of $v$ and $-v$ respectively.
Dancer \cite[Theorem~2]{dancerStructureSolutionsNonlinear1973} proved
that either $C_{\lambda_0}^+$ and $C_{\lambda_0}^-$ are both unbounded or they
intersect away from $(0, \lambda_0)$.

The aim of our paper is to generalise the above results by Rabinowitz and Dancer
from $\Om = \curlX$ to any neighbourhood of $(0, \lambda_0)$. There have already
been some considerations of different domains for odd multiplicity. In his
original paper, Rabinowitz mentioned closures of bounded open sets as a
``weaker'' result \cite[Corollary~1.12]{rabinowitzGlobalResultsNonlinear1971}.
Turner further investigated these domains, proving that if $\del \Om$ is
sufficiently nice and $\lambda_0$ is the only characteristic value $\mu$ with
$(0, \mu) \in \barOm$, then $\del \Om$ admits either two solutions or one solution
of multiplicity two \cite[Theorem~2.4]{turnerTransversalityNonlinearEigenvalue1971}.
His result assumes that $F$ is globally defined,
though. A generalisation of Rabinowitz's result to any $\Om = \overline{\Int \Om}$ has
also been found (presented in \cite[Theorem~29.1]{deimlingNonlinearFunctionalAnalysis1985},
for example); however, this is insufficient, say, when $\Om$ is open. Consequently,
it fails to handle the cases when $H$ has singularities or is unbounded on a
bounded domain. The author is not aware of an existing analogue of the
simple multiplicity result for arbitrary neighbourhoods.

\section{Characteristic values of odd multiplicity}

We start with our generalisation of the odd multiplicity result.
Denote the closure in $\Om$ of non-trivial solutions of $F(x, \lambda) = 0$ by
$\curlS(F)$ and, for any $\lambda_0 \in \Char(K)$ of odd multiplicity, the
connected component of $(0, \lambda_0)$ in $\curlS$ by $C_{\lambda_0}(F)$. We
omit $F$ when it is clear from context. Our aim is to
prove the following theorem for any neighbourhood $\Om$ of $(0, \lambda_0)$.

\begin{thm} \label{thm:global-branches-odd}
    Let $\lambda_0$, $\Om$, $K$ and $F$ be as given in the introduction
    and $\curlS$ be as given above.
    If $\lambda_0 \in \Char(K)$ has odd multiplicity, then the
    connected component $C_{\lambda_0}$ of $(0, \lambda_0)$ in $\curlS$ 
    either is unbounded, admits a limit point on $\del \Om$
    or contains an even number of trivial solutions $(0, \mu)$ of
    $F(x, \lambda) = 0$ with $\mu \in \Char(K)$ of odd multiplicity.
\end{thm}

We remark that all three alternatives for $C_{\lambda_0}$ are possible. A simple
example of the first is $H(x, \lambda) \equiv 0$. The second case is guaranteed
when $\Om$ is a bounded neighbourhood of $(0, \lambda_0)$ such that
$\mu = \lambda_0$ is the only element of $\Char(K)$ with
$(0, \mu) \in \Om$. An instance of the final case can be found in
\cite[pp.~492--493]{rabinowitzGlobalResultsNonlinear1971}.

In the special case that $\Om = \curlX$, the above theorem is simply
\cite[Corollary~1]{dancerStructureSolutionsNonlinear1973}. Our approach
is to reduce the theorem from general $\Om$ to $\Om = \curlX$.
The main step is the following lemma, which will also be useful when we
consider bifurcations at $(0, \lambda_0)$ for simple $\lambda_0$.

\begin{lemma} \label{lem:subset-branches}
    Let $\Om_1 \subseteq \Om_2$ be neighbourhoods of
    $(0, \lambda_0)$ contained in the domain of $F$ and let
    $F_i = F \rvert_{\Om_i}$ for $i = 1,2$. For any closed
    $V \subseteq \curlX$ containing $(0, \lambda_0)$, let $C_V(F_i)$ denote the
    connected component of $V \cap \curlS(F_i)$ containing $(0, \lambda_0)$.
    Then,
    \[
        C_V(F_1) \subseteq C_V(F_2),
    \]
    with equality if $\Om_1$ is closed in $\curlX$ and
    $C_V(F_1) \cap \del \Om_1 = \emptyset$.
\end{lemma}

For this proof and for later results, we need to invoke a special case of a
result by Whyburn \cite[(9.3)]{whyburnTopologicalAnalysis1964}.

\begin{lemma} \label{lem:disjoint-set-connection}
    Let $M$ be a compact metric space. Let $A_1$ and $A_2$ be disjoint closed
    subsets of $M$, with $A_1$ a connected component of $M$. Then there exist
    disjoint compact subsets $M_1$ and $M_2$ of $M$ such that $A_1 \subseteq M_1$,
    $A_2 \subseteq M_2$ and $M = M_1 \cup M_2$.
\end{lemma}

\begin{proof}[Proof of Lemma~\ref{lem:subset-branches}]
    For every $r > 0$, we define
    \[
        \curlX_{\lambda_0}(r) = \{(x, \mu) \in \curlX \mid \|x\| + |\lambda_0 - \mu| < r\}
    \]
    and denote the closure of $\curlX_{\lambda_0}(r)$ in $\curlX$ by
    $\barcurlX_{\lambda_0}(r)$.

    We see that $V \cap \curlS(F_1) \subseteq V \cap \curlS(F_2)$ and so, by
    considering connected components containing $(0, \lambda_0)$, we get
    $C_V(F_1) \subseteq C_V(F_2)$. To prove the equality case,
    suppose that $\Om_1$ is closed in $\curlX$ and
    $C_V(F_1) \cap \del \Om_1 = \emptyset$. Let $N > 0$ be such that
    $C_V(F_1) \subseteq \curlX_{\lambda_0}(N)$.
    We note that bounded closed subsets of $\curlX$ contained in $F^{-1}(0)$
    are compact since $\lambda K x + H(x, \lambda)$ is a compact map, and so
    $\curlS(F_1) \cap \barcurlX_{\lambda_0}(N) \cap V$ is compact.
    By Lemma~\ref{lem:disjoint-set-connection},
    $\curlS(F_1) \cap \barcurlX_{\lambda_0}(N) \cap V$ can be expressed
    as the union of disjoint compact sets $M_1$ and $M_2$ such that
    $C_V(F_1) \subseteq M_1$ and
    $\curlS(F_1) \cap \barcurlX_{\lambda_0}(N) \cap V
        \cap [\del \curlX_{\lambda_0}(N) \cup \del \Om_1] \subseteq M_2$.
    Since $M_1$ and $M_2$ are compact, we can find an open neighbourhood $U$ of
    $M_1$ with $\barU \subseteq \Om_1 \cap \curlX_{\lambda_0}(N)$
    such that $\del U$ and $\curlS(F_1) \cap \barcurlX_{\lambda_0}(N) \cap V$
    are disjoint. We observe that $\del U$ and $\curlS(F_2) \cap V$ are disjoint.

    We see that
    $C_V(F_2)$ is contained in $U$, since if $C_V(F_2)$ intersected
    $\curlX \setminus \barU$ non-trivially, then connectedness would imply that
    $C_V(F_2)$ and $\del U$ are not disjoint. Thus $C_V(F_2)$ must coincide
    with the connected component of $\curlS(F_2) \cap V \cap U$. Since
    \[
        \curlS(F_2) \cap V \cap U = \curlS(F_2) \cap V \cap U \cap \Om_1
        = \curlS(F_1) \cap V \cap U \subseteq \curlS(F_1) \cap V,
    \]
    by looking at the respective connected components containing $(0, \lambda_0)$
    we conclude that $C_V(F_2) \subseteq C_V(F_1)$ and so $C_V(F_1) = C_V(F_2)$.
\end{proof}

Now we are ready to reduce Theorem~\ref{thm:global-branches-odd} to the case
$\Om = \curlX$. We recall that $(0, \lambda_0) \in \Int \Om$ and so have that
the set
\[
    \Om_\delta := \{(x, \lambda) \in \Om \mid \dist((x, \lambda), \del \Om) > \delta\}
\]
is open and non-empty for all $\delta > 0$ sufficiently small.
Let $F_\delta = F \rvert_{\overline{\Om_\delta}}$ for all $\delta > 0$ and
define $h_\delta : \overline{\Om_\delta} \cup (\{0\} \times \R) \to X$ by
\[
    h_\delta(x, \lambda) =
    \begin{cases}
        \|x\|^{-1} H(x, \lambda), & x \neq 0, \\
        0, & x = 0.
    \end{cases}
\]
By our assumption on $H$, we have that $h_\delta$ is continuous.
Let $\widetilde{h}_\delta$ be an extension of $h_\delta$ to
$\curlX$ as given by Dugundji's extension theorem
\cite[Chapter~IX Theorem~6.1]{dugundjiTopology1978}. For all
$(x, \lambda) \in \curlX$, let
$\tildeH_\delta(x, \lambda) = \|x\|\widetilde{h}_\delta(x, \lambda)$
and $\tildeF_\delta(x, \lambda) = x - \lambda Kx - \tildeH_\delta(x, \lambda)$.

Assume that the proposition holds when $\Om = \curlX$ and 
suppose that $C_{\lambda_0}(F)$ is bounded with no accumulation points
on $\del \Om$. Then, since $C_{\lambda_0}(F)$ is compact and disjoint from
$\del \Om$, for some $\delta > 0$ sufficiently small we have
$C_{\lambda_0}(F) \subseteq \Om_{\delta}$. Applying
Lemma~\ref{lem:subset-branches} with $\Om_1 = \barOm_{\delta}$, $\Om_2 = \Om$
and $V = \curlX$, we get that
$C_{\lambda_0}(F_{\delta}) \subseteq C_{\lambda_0}(F) \subseteq \Om_{\delta}$
and so $C_{\lambda_0}(F_{\delta}) \cap \del \Om_{\delta} = \emptyset$.
Now by applying Lemma~\ref{lem:subset-branches} twice, both times with $\Om_1$ and
$V$ as before, once with $F$ and $\Om_2 = \Om$, and once with $\tildeF_\delta$
and $\Om_2 = \curlX$, we obtain
\[
    C_{\lambda_0}(F) = C_{\lambda_0}(F_\delta) = C_{\lambda_0}(\tildeF_\delta).
\]
Consequently, since
$C_{\lambda_0}(F)$ is bounded and $\tildeF_\delta$
is defined on all of $\curlX$, we may apply
Theorem~\ref{thm:global-branches-odd} to get that
$C_{\lambda_0}(F) = C_{\lambda_0}(\tildeF_\delta)$ contains an even number of trivial
solutions $(0, \mu)$ of $\tildeF_\delta(x, \lambda) = 0$ with $\mu$ of odd
multiplicity. Since $C_{\lambda_0}(F) \subseteq \Om_\delta$ and
$F = \tildeF_\delta$ on $\barOm_\delta$, we conclude
that $C_{\lambda_0}(F)$ contains an even number of trivial
solutions $(0, \mu)$ of $F(x, \lambda) = 0$ with $\mu$ of odd
multiplicity. Thus we have reduced Theorem~\ref{thm:global-branches-odd} to
the known case $\Om = \curlX$.

\section{Simple characteristic values}

Now we consider the special case where $\lambda_0$ is a simple
characteristic value. We start by giving the definition of a branch of solutions
in the direction of $v$ or $-v$, where $v$ is a unit length
$\lambda_0^{-1}$-eigenvector of $K$. Let $X'$ be the dual space of $X$,
and let $l \in X'$ be the $\lambda_0^{-1}$-eigenvector of the dual of
$K$ such that $\langle l, v \rangle = 1$. For all $0 \leq y < 1$, define
\[
    \curlC_y = \{(x, \lambda) \in \curlX \mid |\langle l, x \rangle| > y\|x\|\}.
\]
Let $\curlC_y^+$ and $\curlC_y^-$ be the subsets of $\curlC_y$ consisting of the
elements with $\langle l, x \rangle > y\|x\|$ and $\langle l, x \rangle < -y\|x\|$,
respectively. We say that $F(x, \lambda) = 0$ admits a
branch of solutions at $(0, \lambda_0)$ in the direction
of $v$ if there exists a connected set $Q^+ \subseteq C_{\lambda_0}$
containing $(0, \lambda_0)$ such that for every $y \in (0, 1)$, there exists
$\epsilon_y^+ > 0$ for which
\[
    \emptyset \neq Q^+ \cap \del \curlX_{\lambda_0}(\epsilon) \subseteq \curlC_y^+
\]
for all $0 < \epsilon < \epsilon_y$.
We then call $Q^+$ a branch of solutions in the direction of $v$. We replace $v$
with $-v$ and swap $+$ with $-$ to get the definition of a branch of solutions
in the direction of $-v$. 

Denote by $C_{\lambda_0}^+$ and $C_{\lambda_0}^-$ the closures in $\Om$ of the
unions of all branches of solutions in the directions of $v$ and $-v$ respectively.
Our desired result is the following Theorem.

\begin{thm} \label{thm:simple-bifurcation-branches}
    Let $\lambda_0$, $\Om$, $K$ and $F$ be as given in the introduction. Suppose
    that $\lambda_0$ is a simple characteristic value, $v$ is a unit length
    $\lambda_0^{-1}$-eigenvector of $K$, and $C_{\lambda_0}^+$ and $C_{\lambda_0}^-$
    are the closures in $\Om$ of the unions of all branches of solutions
    of \eqref{eqn:nonlinear-eigenvalue-problem}
    in the directions of $v$ and $-v$ respectively. Then at least one
    of the following alternatives holds:
    \begin{enumerate}
        \item each of $C_{\lambda_0}^+$ and $C_{\lambda_0}^-$ is unbounded or
              admits a limit point on $\del \Om$;
        \item $C_{\lambda_0}^+ \cap C_{\lambda_0}^- \neq \{(0, \lambda_0)\}$.
    \end{enumerate}
\end{thm}

Similarly to Theorem~\ref{thm:global-branches-odd}, both alternatives of
Theorem~\ref{thm:simple-bifurcation-branches} can occur. Moreover, the first
alternative cannot be strengthened to say that $C_{\lambda_0}^+$ and
$C_{\lambda_0}^-$ are both unbounded or both admit an accumulation point on
$\del \Om$. A counter-example is $H(x, \lambda) \equiv 0$ on domain
$\Om = \overline{\curlC_0^+} \cup (\overline{\curlC_0^-} \cap \curlX_{\lambda_0}(N))$,
for any $N > 0$. In this case, $x \mapsto F(x, \lambda)$ is a linear map for every
fixed $\lambda$, with kernel the $\lambda^{-1}$-eigenspace of $K$ for
$\lambda \neq 0$. It follows that $C_{\lambda_0}^-$ is bounded,
$C_{\lambda_0}^+$ does not have an accumulation point on $\del \Om$ and
$C_{\lambda_0}^+ \cap C_{\lambda_0}^- = \{(0, \lambda_0)\}$.

To avoid duplication, we will use $\kappa$ to denote one of
$+$ and $-$ and will interpret $-\kappa$ in the obvious way. Fix $0 < y < 1$.
By \cite[Lemma~1.2]{rabinowitzGlobalResultsNonlinear1971},
there exists $S > 0$ such that $\barcurlX_{\lambda_0}(S) \subseteq \Int \Om$
and
\begin{equation} \label{eqn:local-cone}
    \barcurlX_{\lambda_0}(S) \cap \curlS \setminus \{(0, \lambda_0)\}
        \subseteq \curlC_y.
\end{equation}
For every $0 < \epsilon < S$, let $C_{\lambda_0,\epsilon}^\kappa$ be
the connected component of
$C_{\lambda_0} \setminus (\curlX_{\lambda_0}(\epsilon) \cap \curlC_y^{-\kappa})$
containing $(0, \lambda_0)$.
We notice that $C_{\lambda_0,\epsilon}^\kappa \supseteq C_{\lambda_0,\epsilon'}^\kappa$
for all $0 < \epsilon < \epsilon' < S$, and so
\[
    \bigcup_{0 < \epsilon < S} C_{\lambda_0,\epsilon}^\kappa
    = \bigcup_{0 < \epsilon < \epsilon'} C_{\lambda_0,\epsilon}^\kappa
\]
for all $0 < \epsilon' < S$.
Also, we deduce from \eqref{eqn:local-cone} that, regardless of $y$, every branch
of solutions is contained in some $C_{\lambda_0,\epsilon}^\kappa$.
Thus $C_{\lambda_0}^\kappa$ is the closure of
$\bigcup_{0 < \epsilon < S} C_{\lambda_0,\epsilon}^\kappa$.
We note that $C_{\lambda_0}^\kappa$ is connected as the closure of a union of
connected sets sharing a point
\cite[Chapter~V Theorem~1.5~\&~1.6]{dugundjiTopology1978}.

Rather than proving directly that $C_{\lambda_0}^+$ and $C_{\lambda_0}^-$
satisfy at least one of the alternatives in
Theorem~\ref{thm:simple-bifurcation-branches}, we will show the following
stronger result.

\begin{prop} \label{prop:simple-bifurcation-special}
    If $C_{\lambda_0}^\kappa$ is bounded and disjoint from $\del \Om$ for
    $\kappa \in \{\pm\}$, then the connected component $T_\epsilon^\kappa$ of
    $C_{\lambda_0}^\kappa \cap \barcurlX_{\lambda_0}(\epsilon) \cap \overline{\curlC_y^{-\kappa}}$
    containing $(0, \lambda_0)$ intersects $\del \curlX_{\lambda_0}(\epsilon)$
    non-trivially for all $\epsilon > 0$ sufficiently small.
\end{prop}

We verify that this proposition does in fact imply
Theorem~\ref{thm:simple-bifurcation-branches}. If the first alternative of the
theorem does not hold, then $C_{\lambda_0}^\kappa$, and so $T_\epsilon^\kappa$
for all $0 < \epsilon < S$, is bounded
and disjoint from $\del \Om$ for some $\kappa \in \{\pm\}$.
Since $T_\epsilon^\kappa$
is connected and $(0, \lambda_0) \in T_\epsilon^\kappa$, 
by definition of $C_{\lambda_0, \epsilon}^{-\kappa}$ we have
$T_\epsilon^\kappa \subseteq C_{\lambda_0, \epsilon}^{-\kappa} \subseteq C_{\lambda_0}^{-\kappa}$.
From the proposition we have that $T_\epsilon^\kappa$ intersects
$\del \curlX_{\lambda_0}(\epsilon)$ non-trivially for all $\epsilon > 0$
sufficiently small, and so we conclude that the
second alternative of the theorem holds.

To prove Proposition~\ref{prop:simple-bifurcation-special}, we first reduce it
to the case $\Om = \curlX$.
Assume that the proposition holds when $\Om = \curlX$. Recall the definitions
of $\Om_\delta$, $F_\delta$ and $\tildeF_\delta$ from earlier.
Suppose that $C_{\lambda_0}^\kappa(F)$ is bounded without any accumulation 
points on $\del \Om$.
By adapting the reduction argument of Theorem~\ref{thm:global-branches-odd} to use
$C_{\lambda_0}^\kappa(F)$ instead of $C_{\lambda_0}(F)$, we get
$C_{\lambda_0}^\kappa(\tildeF_\delta) = C_{\lambda_0}^\kappa(F)$ for
all $\delta > 0$ sufficiently small.
We apply Proposition~\ref{prop:simple-bifurcation-special} to $\tildeF_\delta$
to obtain that the connected component of
\[
    C_{\lambda_0}^\kappa(\tildeF_\delta) \cap \barcurlX_{\lambda_0}(\epsilon)
        \cap \overline{\curlC_y^{-\kappa}}
    = C_{\lambda_0}^\kappa(F) \cap \barcurlX_{\lambda_0}(\epsilon)
        \cap \overline{\curlC_y^{-\kappa}}
\]
containing $(0, \lambda_0)$ intersects $\del \curlX_{\lambda_0}(\epsilon)$ for
all $\epsilon$ sufficiently small. Thus we have reduced the proposition to the
case $\Om = \curlX$.

Finally, we adapt the proof of \cite[Theorem~2]{dancerStructureSolutionsNonlinear1973}
to show that Proposition~\ref{prop:simple-bifurcation-special} holds when
$\Om = \curlX$. We will invoke the following result due to Dancer
\cite[Lemma~3]{dancerStructureSolutionsNonlinear1973}.

\begin{prop} \label{prop:simple-continuum}
    Let $\Om = \curlX$ and take $S$ as in \eqref{eqn:local-cone}. 
    Then for every $0 < \epsilon < S$, the set
    $C_{\lambda_0,\epsilon}^\kappa$ either is unbounded or intersects
    $\del \curlX_{\lambda_0}(\epsilon) \cap \curlC_y^{-\kappa}$ non-trivially.
\end{prop}

We also need the following lemma.

\begin{lemma}  \label{lem:connected-sets}
    Let $\kappa \in \{\pm\}$.
    For every $0 < \epsilon_1 < S$, if
    $C_{\lambda_0,\epsilon_1}^\kappa \cap \del \curlX_{\lambda_0}(\epsilon_1)
        \cap \overline{\curlC_y^{-\kappa}} \neq \emptyset$ then the set
    \[
        \biggl( C_{\lambda_0}^\kappa \cap \curlX_{\lambda_0}(\epsilon_1) \cap \curlC_y^{-\kappa} \biggr)
            \cup \biggl( \del \curlX_{\lambda_0}(\epsilon_1) \cap \overline{\curlC_y^{-\kappa}} \biggr)
    \]
    is connected.
\end{lemma}

\begin{proof}
    Fix $0 < \epsilon_1 < S$ and let
    $Y = (C_{\lambda_0}^\kappa \cap \curlX_{\lambda_0}(\epsilon_1) \cap \curlC_y^{-\kappa})
        \cup (\del \curlX_{\lambda_0}(\epsilon_1) \cap \overline{\curlC_y^{-\kappa}})$.
    To prove that $Y$ is connected, we only need to show that
    \[
        Y_\epsilon := \biggl(C_{\lambda_0,\epsilon}^\kappa \cap \curlX_{\lambda_0}(\epsilon_1) \cap \curlC_y^{-\kappa} \biggr)
            \cup \biggl(\del \curlX_{\lambda_0}(\epsilon_1) \cap \overline{\curlC_y^{-\kappa}} \biggr)
    \]
    is connected for all $0 < \epsilon < \epsilon_1$. Then
    $\bigcup_{0 < \epsilon < S} Y_\epsilon$ is connected
    as the union of connected sets sharing a point
    \cite[Chapter~V Theorem~1.5]{dugundjiTopology1978}. Since
    $A \cap \overline{B} \subseteq \overline{A \cap B}$
    for all sets $A$ and $B$ with $A$ open, by taking
    $A = \curlX_{\lambda_0}(\epsilon_1) \cap \curlC_y^{-\kappa}$ and
    $B = \bigcup_{0 < \epsilon < \epsilon_1} C_{\lambda_0,\epsilon}^\kappa$ we see that
    \[
        \bigcup_{0 < \epsilon < \epsilon_1} Y_\epsilon
        \subseteq Y
        \subseteq \overline{\left( \left( \bigcup_{0 < \epsilon < \epsilon_1} C_{\lambda_0,\epsilon}^\kappa \right)
            \cap \curlX_{\lambda_0}(\epsilon_1) \cap \curlC_y^{-\kappa} \right)}
                \cup \biggl(\del \curlX_{\lambda_0}(\epsilon_1) \cap \overline{\curlC_y^{-\kappa}} \biggr)
        = \overline{\bigcup_{0 < \epsilon < \epsilon_1} Y_\epsilon}
    \]
    and so $Y$ is connected as a set contained between a connected set and its
    closure \cite[Chapter~V Theorem~1.6]{dugundjiTopology1978}.

    Now we show that $Y_\epsilon$ is connected.
    Let $V$ be a closed and open subset of $Y_\epsilon$ for some
    $0 < \epsilon < \epsilon_1$ fixed. Since
    $\del \curlX_{\lambda_0}(\epsilon_1) \cap \overline{\curlC_y^{-\kappa}}$ is connected,
    $V$ is either disjoint from it or contains it. Swapping $V$ with its
    complement in $Y_\epsilon$ if needed, we may assume that the former case
    is true. Thus we have that $V$ is a subset of $C_{\lambda_0,\epsilon}^\kappa$.
    We see from
    $\curlS \setminus \{(0, \lambda_0)\} \cap \barcurlX_{\lambda_0}(\epsilon_1) \subseteq \curlC_y$
    and the definition of $C_{\lambda_0,\epsilon}^\kappa$ that
    \[
        C_{\lambda_0,\epsilon}^\kappa \cap \barcurlX_{\lambda_0}(\epsilon_1) \cap \curlC_y^{-\kappa}
        = C_{\lambda_0,\epsilon}^\kappa \cap \barcurlX_{\lambda_0}(\epsilon_1)
            \cap \overline{\curlC_y^{-\kappa}} \setminus \curlX_{\lambda_0}(\epsilon)
    \]
    and so $V$ is a closed subset of a closed set in $\curlX$. Also, $V$ is open
    in $C_{\lambda_0,\epsilon}^\kappa$ since $V$ is open in
    $C_{\lambda_0,\epsilon}^\kappa \cap \curlX_{\lambda_0}(\epsilon_1) \cap \curlC_y^{-\kappa}$, an open subset of
    $C_{\lambda_0,\epsilon}^\kappa$. From the connectedness of $C_{\lambda_0,\epsilon}^\kappa$, we obtain that
    either $V = \emptyset$ or $V = C_{\lambda_0,\epsilon}^\kappa$.
    Since $V$ is disjoint from
    $\del \curlX_{\lambda_0}(\epsilon_1) \cap \overline{\curlC_y^{-\kappa}}$,
    from
    $C_{\lambda_0,\epsilon_1}^\kappa \cap \del \curlX_{\lambda_0}(\epsilon_1)
    \cap \overline{\curlC_y^{-\kappa}} \neq \emptyset$ we get that
    $V = \emptyset$ and so $Y_\epsilon$ is connected.
\end{proof}

Now we can prove Proposition~\ref{prop:simple-bifurcation-special},
and so conclude that Theorem~\ref{thm:simple-bifurcation-branches} holds.

\begin{proof}[Proof of Proposition~\ref{prop:simple-bifurcation-special}]
    We justified earlier that this proposition reduces to the case $\Om = \curlX$,
    and so we assume $\Om = \curlX$. Suppose that $C_{\lambda_0}^\kappa$ is bounded.
    Suppose for a contradiction that $T_\epsilon^\kappa$ is disjoint from
    $\del \curlX_{\lambda_0}(\epsilon)$ for some $\epsilon \in (0, S)$. By
    Lemma~\ref{lem:disjoint-set-connection}
    we can express
    $C_{\lambda_0}^\kappa \cap \barcurlX_{\lambda_0}(\epsilon) \cap \overline{\curlC_y^{-\kappa}}$
    as the disjoint union of compact sets $M_1$ and $M_2$ with
    $T_\epsilon^\kappa \subseteq M_1$ and
    $C_{\lambda_0}^\kappa \cap \del \curlX_{\lambda_0}(\epsilon) \cap \overline{\curlC_y^{-\kappa}} \subseteq M_2$.
    We see that $M_1 \subseteq \curlX_{\lambda_0}(\epsilon)$ and
    so, by compactness of $M_1$, there exists
    an open neighbourhood of $M_1$ contained in
    $\curlX_{\lambda_0}(\epsilon')$ for some $\epsilon' \in (0, \epsilon)$,
    with boundary disjoint from
    $C_{\lambda_0}^\kappa \cap \curlX_{\lambda_0}(\epsilon) \cap \overline{\curlC_y^{-\kappa}}$.
    To obtain a contradiction, we show that for every
    $0 < \epsilon_0 < \epsilon_1 < S$,
    the boundary of every open neighbourhood of $(0, \lambda_0)$
    contained in $\curlX_{\lambda_0}(\epsilon_0)$ intersects
    $C_{\lambda_0}^\kappa \cap \curlX_{\lambda_0}(\epsilon_1) \cap \overline{\curlC_y^{-\kappa}}$
    non-trivially.

    Fix $0 < \epsilon_0 < \epsilon_1 < S$ and let
    $U \subseteq \curlX_{\lambda_0}(\epsilon_0)$ be an open neighbourhood of
    $(0, \lambda_0)$. We know that $C_{\lambda_0, \epsilon}^\kappa \subseteq C_{\lambda_0}^\kappa$
    for all $0 < \epsilon < S$, and so Proposition~\ref{prop:simple-continuum}
    yields that
    \[
        \emptyset
        \neq C_{\lambda_0,\epsilon}^\kappa \cap \del \curlX_{\lambda_0}(\epsilon) \cap \curlC_y^{-\kappa}
        \subseteq C_{\lambda_0}^\kappa \cap \del \curlX_{\lambda_0}(\epsilon) \cap \curlC_y^{-\kappa}.
    \]
    Using this and
    $C_{\lambda_0}^\kappa \cap \del \curlX_{\lambda_0}(\epsilon_1) \subseteq \curlX \setminus U$,
    we see that
    \[
        Y := \biggl(C_{\lambda_0}^\kappa \cap \curlX_{\lambda_0}(\epsilon_1) \cap \curlC_y^{-\kappa} \biggr)
            \cup \biggl(\del \curlX_{\lambda_0}(\epsilon_1) \cap \overline{\curlC_y^{-\kappa}} \biggr)
    \]
    intersects $\curlX \setminus U$ non-trivially and
    $(0, \lambda_0) \in \overline{Y} \cap U$.
    From Lemma~\ref{lem:connected-sets} we get that $Y$, and so $\overline{Y}$,
    is connected. Thus $\del U \cap \overline{Y} \neq \emptyset$ and so, since
    $U \subseteq \curlX_{\lambda_0}(\epsilon_0)$, we conclude that $\del U$ intersects
    $C_{\lambda_0}^\kappa \cap \curlX_{\lambda_0}(\epsilon_1) \cap \overline{\curlC_y^{-\kappa}}$
    non-trivially, as required.
\end{proof}

\section*{Acknowledgement}

I discovered the content of this paper while writing my Honours thesis.
I would like to thank Daniel Daners for his supervision and for encouraging me
to write this paper.

\printbibliography

\end{document}